\documentclass[12pt,oneside,a4paper]{amsart}

\usepackage{silence}
\usepackage{subfiles}
\usepackage[T1]{fontenc}
\usepackage[utf8]{inputenc}
\usepackage[dvipsnames]{xcolor}
\usepackage[yyyymmdd,hhmmss]{datetime}
\usepackage[numbers,sort,compress]{natbib}
\usepackage{amstext,amsthm,amssymb}
\usepackage{xparse,mathrsfs,mathtools}
\usepackage[british]{babel}
\usepackage[babel=true,english=british]{csquotes}
\usepackage{xr-hyper}
\usepackage{stmaryrd}
\usepackage{enumitem}
\usepackage[pdftex,unicode=true,%
	pdfusetitle,pdfdisplaydoctitle=true,%
	pdfpagemode=UseOutlines,%
	bookmarks=true,bookmarksnumbered=false,bookmarksopen=true,bookmarksopenlevel=2,%
	breaklinks=true,%
	pdfencoding=unicode,psdextra,
	pdfcreator={},
	pdfborder={0 0 0},
	colorlinks=false
]{hyperref}
\usepackage{lmodern,microtype}
\usepackage{geometry}
\usepackage[capitalise,nameinlink]{cleveref}
\usepackage[cal=zapfc,bb=ams,frak=esstix,scr=boondoxo]{mathalfa}

\hypersetup{
	colorlinks=false,
	pdfborder={0 0 0},
	pdfborderstyle={/S/U/W 0},
}

\newcommand*{\doi}[1]{\href{http://dx.doi.org/#1}{doi: {\texttt{\footnotesize\color{teal}#1}}}}

\newtheorem*{thm*}{Theorem}
\newtheorem{thm}{Theorem}[section]
\newtheorem{lem}[thm]{Lemma}

\newtheorem{prop}[thm]{Proposition}

\theoremstyle{remark}

\newtheorem*{rem*}{Remark}

\numberwithin{equation}{section}
\crefformat{equation}{#2(#1)#3}
\crefformat{enumi}{#2(#1)#3}
\crefname{section}{\textsection}{\textsection\textsection}
\newcommand{\sawtooth}[1]{(\mkern -2mu(#1)\mkern -2mu)}
\newcommand{\compl}{\mathrm{c}}
\newcommand{\NN}{\mathbb{N}}
\newcommand{\ZZ}{\mathbb{Z}}
\newcommand{\QQ}{\mathbb{Q}}
\newcommand{\RR}{\mathbb{R}}

\newcommand{\LandauO}{O}

\DeclareMathOperator{\inv}{inv}
\DeclarePairedDelimiter\parentheses{\lparen}{\rparen}
\DeclarePairedDelimiter\braces{\lbrace}{\rbrace}
\DeclarePairedDelimiter\brackets{\lbrack}{\rbrack}
\DeclarePairedDelimiter\abs{\lvert}{\rvert}

\DeclarePairedDelimiter\floor{\lfloor}{\rfloor}
\DeclarePairedDelimiter\ceil{\lceil}{\rceil}
\DeclarePairedDelimiter\ropeninterval{\lbrack}{\rparen}
\DeclarePairedDelimiter\lopeninterval{\lparen}{\rbrack}
\DeclarePairedDelimiter\rcfrac{\llbracket}{\rrbracket} 
\NewDocumentCommand\set{ s o m o }{%
	\IfBooleanTF{#1}{\IfNoValueTF{#4}{\braces*{#3}}{\braces*{\,#3:#4\,}}}{%
		\IfNoValueTF{#2}{\IfNoValueTF{#4}{\braces{#3}}{\braces{\,#3:#4\,}}}{%
			\IfNoValueTF{#4}{\braces[#2]{#3}}{\braces[#2]{\,#3:#4\,}}}}%
}
\NewDocumentCommand\e{ s O{} m }{%
	\IfBooleanTF{#1}{%
		\operatorname{e}_{#2}\parentheses*{#3}%
	}{\operatorname{e}_{#2}\parentheses{#3}}%
}
\newcommand{\Sum}{\sideset{}{^*}{\sum}}

\title{On restricted averages of Dedekind sums}
\date{\today{}}



\author{Paolo~Minelli}
\author{Athanasios~Sourmelidis}
\author{Marc~Technau}

\address{%
	Paolo~Minelli\\%
	Department of Mathematics\\%
	KTH Royal Institute of Technology\\%
	Stockholm\\%
	Sweden%
	\and%
	Athanasios~Sourmelidis%
	\and%
	Marc~Technau\\%
	Insitute for Analysis and Number Theory\\%
	Graz University of Technology\\%
	Kopernikusgasse~24/II\\%
	8010~Graz\\%
	Austria%
}

\email{pminelli@kth.se}
\email{mtechnau@math.tugraz.at}
\email{sourmelidis@math.tugraz.at}

\thanks{%
	During the preparation of this article, PM was supported  by the Austrian Science Fund (FWF), project~I-3466, by project F-5510-N26 (which is part of the Special Research program `Quasi-Monto Carlo Methods: Theory and Applications') and by the Knut and Alice Wallenberg foundation grant KAW 2019.0503.
	AS is supported by FWF project~M~3246-N.
	MT is supported by the joint FWF--ANR project \emph{ArithRand} (FWF I 4945-N and ANR-20-CE91-0006).
}

\begin{document}
\begin{abstract}
	We investigate the averages of Dedekind sums over rational numbers in the set 
	$\mathscr{F}_\alpha(Q)\coloneqq\{\,{v}/{w}\in \mathbb{Q}: 0<w\leq Q\,\}\cap \lbrack 0, \alpha \rparen$
	for fixed $\alpha\leq 1/2$. 
	In previous work, we obtained asymptotics for $\alpha=1/2$, confirming a conjecture of Ito in a quantitative form. 
	In the present article we extend our former results, first to all fixed rational $\alpha$ and then to almost all irrational $\alpha$. 
	As an intermediate step we obtain a result quantifying the bias occurring in the second term of the asymptotic for the average running time of the \textit{by-excess} Euclidean algorithm, which is of independent interest.
\end{abstract}
\maketitle%
\section{Introduction and main results}

In this work we are concerned with distributional questions regarding Dedekind sums and related questions in connexion with continued fraction expansions of rational numbers.

\subsection{Dedekind sums}

Dedekind sums have their genesis in the multiplier system for Dedekind's $\eta$ function with respect to the modular group $\operatorname{SL}_2(\ZZ) / \set{ \pm\parentheses{\!{\begin{smallmatrix} 1 & \\ & 1 \end{smallmatrix}}\!} }$, first appearing in Dedekind's supplements to Riemann's collected works ~\cite{Dedekind}.
A quick way of introducing them proceeds as follows.
Let $\floor{x}$ denote the integer part of $x\in\RR$. 
If we denote by $\sawtooth{\,\cdot\,}$ the \emph{saw-tooth function},
\[
\sawtooth{x} = \begin{cases}
x - \floor{x} - \frac{1}{2} & \text{if } x\in\RR\setminus\ZZ, \\
0                           & \text{if } x\in\ZZ,             
\end{cases}
\]
then, for any pair $(a,b)\in\ZZ^2$, $b\neq 0$, the \emph{Dedekind sum}
$s(a,b)$ is defined by
\[
s(a,b) = \sum_{n=1}^b \parentheses*{\!\!\parentheses*{\frac{n}{b}}\!\!} \parentheses*{\!\!\parentheses*{\frac{na}{b}}\!\!}.
\]
It can be verified that $s(a,b) = s(ka,kb)$ for any non-zero integer $k$.
Hence, $s(a/b) \coloneqq s(a,b)$ yields a well-defined function on $\mathbb{Q}_{\geq 0}$.

Here we are concerned with averages of Dedekind sums.
Using the symmetry property $s(1-x) = - s(x)$, it is easy to see that
\begin{equation}\label{eq:SumOfDedekindSums}
\sum_x s(x) = 0,
\end{equation}
when $x$ ranges over the fractions in $[0,1]$ with fixed denominator.
The same holds true, when the sum is extended over all fractions in $[0,1]$ with bounded denominator.
On the other hand, when restricting to the fractions in $[0,1/2]$, then Ito~\cite{ito2004density-result} conjectured on the basis of numerical evidence that the above sum, divided by the number of terms, diverges to $+\infty$.
The main result of~\cite{MinelliSourmelidisTechnau} is a proof of a quantitative version of this conjecture.
To state it, let
\[
\mathscr{F}(Q) \coloneqq \set*{ \frac{v}{w}\in \QQ\cap[0,1] }[ w\leq Q ]
\]
denote the set of \emph{Farey fractions} of order $Q$.
Moreover, for any $\alpha\in(0,1)$, let
\[
\mathscr{F}_{\alpha}(Q) \coloneqq \mathscr{F}(Q) \cap \ropeninterval{0,\alpha}, \quad
\mathscr{F}_{\alpha}^\compl(Q) \coloneqq \mathscr{F}(Q) \cap \lopeninterval{1-\alpha,1}.
\]
The main result of~\cite{MinelliSourmelidisTechnau} then states that
\begin{align}\label{eq:Ito:AsymptoticVersion1/2}
	\frac{1}{\sharp \mathscr{F}(Q)} \sum_{x\in\mathscr{F}_{1/2}(Q)} s(x)
	= \frac{1}{16} \log Q + O(1).
\end{align}

Our main goal here is to prove the following generalisation of~\cref{eq:Ito:AsymptoticVersion1/2}, allowing for \textit{arbitrary cuts} $\alpha$ in place of $1/2$.
For rational $\alpha$ we have the following result:
\begin{thm}\label{thm:Goal}
	Let $\alpha={v}/{w}$ be a reduced rational number in $(0,1/2]$. 
	Then
	\[
	\frac{1}{\sharp \mathscr{F}(Q)} \sum_{x\in\mathscr{F}_{\alpha}(Q)} s(x)
	= \frac{1}{12} \parentheses*{ 1 - \frac{1}{w^2} } \log Q+\LandauO_\alpha(1).
	\]
\end{thm}

The shape of the leading term in~\cref{eq:Ito:AsymptoticVersion1/2} raises the question as to how the corresponding results for \emph{irrational} $\alpha$ might look.
We do not resolve this here in full generality, as irrational $\alpha$ which are well approximable by rationals cause problems in our analysis.
Nevertheless, we can prove the following partial result:
\begin{thm}\label{thm:ourresult2}
	There is a subset $\mathscr{S}\subseteq[0,1/2]$ of full Lebesgue measure such that for every $\alpha\in\mathscr{S}$ one has the asymptotic formula
	\[
	\frac{1}{\sharp \mathscr{F}(Q)} \sum_{x\in\mathscr{F}_{\alpha}(Q)} s(x)
	= \frac{1}{12} \log Q
	+ o_{\alpha}(\log Q)
	\quad(\text{as } Q\to\infty).
	\]
\end{thm}

\subsection{Asymptotics for the number of steps of Euclidean algorithms}

There is a well-known link between Dedekind sums, continued fraction expansions and Euclidean algorithms.
To sketch this, we require some notation.
Any rational number $x\in\ropeninterval{0,1}$ admits a \emph{continued fraction expansion}
\begin{equation}\label{eq:continuedfractionexpansion}
x
= [0; a_1, a_2, \ldots, a_n]
\coloneqq 0 + \cfrac{1}{
	a_1 + \cfrac{1}{
		a_2 + \ldots \cfrac{}{
			\ldots + \cfrac{1}{
				a_n
			}
		}
	}
} 
\end{equation}
with \emph{partial quotients} $a_1,a_2,\ldots,a_n \in \NN$.
When requiring that $a_n\neq 1$, which we shall do throughout the sequel, it is well known that $n$ and the partial quotients are uniquely determined by $x$.
The number $n$, for which we also write $L(x)$, admits an interpretation as the number of steps of a certain variant of the Euclidean algorithm, whereas the sum $a_1 + a_2 + \ldots + a_n$ of partial quotients has a similar interpretation.
We also let
\[
\varSigma_{\mathrm{odd} }(x) = \sum_{\substack{ j = 1 \\ j \text{ odd}  }}^n a_j, \quad
\varSigma_{\mathrm{even}}(x) = \sum_{\substack{ j = 2 \\ j \text{ even} }}^n a_j, \quad 
\varSigma_{\pm}(x) = \Sigma_{\mathrm{odd} }(x) - \Sigma_{\mathrm{even} }(x).
\]
The link between Dedekind sums and continued fraction expansions is furnished by a formula due to Barkan~\cite{Barkan} and Hickerson~\cite{hickerson1977continued-fractions}:
\begin{equation}\label{eq:hickersonformula}
s(x)
= \frac{(-1)^n - 1}{8} + \frac{
	x
	- (-1)^n [0;a_n,\ldots,a_2,a_1]
	+ \varSigma_\pm(x)
}{12}.
\end{equation}
Consequently, the study of sums such as~\cref{eq:SumOfDedekindSums} is intimately linked with distribution properties of continued fractions or statistical properties of the number of steps of certain Euclidean algorithms.

The study of such properties has a long history and dates back to the seminal work of
Lochs~\cite{lochs1961statistik} and Heilbronn~\cite{heilbronn1968average-length},
who identified the principal term of the asymptotics for the average number of steps of the classical Euclidean algorithm
\[
\frac{1}{\varphi(b)} \sum_{\substack{ a\leq b \\ \gcd(a,b)=1 }} L\parentheses*{\frac{a}{b}}= A_1 \log b + O\parentheses*{ (\log\log b)^4 } \quad(\text{as~} b\to\infty);
\]
here $\varphi(n)$ is the Euler totient function, $A_1$ is an explicitly given non-zero constant and the error term is absolute.
For the same average, an asymptotic formula with two significant terms was obtained later by Porter~\cite{Porter}.
These results are with respect to averages over numerators and fixed denominator.

Concerning averages over both numerators {and} denominators, an asymptotic formula with power-law fall-off in the error term was obtained by Vall\'ee~\cite{vallee2000a-unifying-framework}. 
This was improved by Ustinov~\cite{Ustinovasymptotic}, who obtained an asymptotic formula with better fall-off in the error term than the one that can be derived from Porter's result
\[
\frac{1}{\sharp \mathscr{F}(Q)}\sum_{x\in\mathscr{F}(Q)}  L\parentheses*{x}= B_1 \log Q + B_2 + O\parentheses*{ (\log Q)^5 / Q },
\]
where $B_1$ and $B_2$ are explicitly given non-zero constants.
While examining the statistical properties of different variants of the Euclidean algorithm, Vall\'ee~\cite{Valleedynamical} obtained also the leading term of the asymptotic formula for the expectation of the number of steps of the so-called \emph{by-excess} Euclidean algorithm which generates for a given rational number $a/b\in(0,1]$ the minus continued fraction expansion
\[
\frac{a}{b}
= \rcfrac{1; b_1, b_2, \ldots, b_n}
\coloneqq 1 - \cfrac{1}{
	b_1 - \cfrac{1}{
		b_2 - \ldots \cfrac{}{
			\ldots - \cfrac{1}{
				b_n
			}
		}
	}
} \, .
\]
Zhabitskaya~\cite{zhabitskaya2009average-length} improved upon Vall\'ee's result by showing that if $\ell(a/b)=m$ denotes the number of partial quotients in the minus continued fraction expansion of $a/b$, then
\[
\frac{1}{\sharp \mathscr{F}(Q)} \sum_{x\in\mathscr{F}(Q) } \ell\parentheses*{x}=C_1 (\log Q)^2 + C_2 \log Q + C_3 + O\parentheses*{ (\log Q)^6 / Q };
\]
here $C_1, C_2, C_3$ are explicitly given non-zero constants, the first two being given by
\begin{equation}\label{eq:Zhabitskaya:Constants}
C_1 = \frac{1}{2\zeta(2)}, \quad C_2 = \frac{1}{\zeta(2)} \parentheses*{ 2\gamma - \frac{3}{2} - 2 \frac{\zeta'(2)}{\zeta(2)} }
\end{equation}
and $\gamma$ denoting the Euler--Mascheroni constant.
For more results regarding the expectation and the variance of the number of steps of the classical and by-excess Euclidean algorithm, we also refer to the work of Baladi and Vall\'ee~\cite{BaladiVallee}, Bykovski\u{\i}~\cite{Bykovski}, Dixon~\cite{dixon1970number-of-steps,Dixonasimple}, Hensley~\cite{hensley1994number-of-steps} and Ustinov~\cite{Ustinovcalculation,Ustinovonthestatistical}.

The approach to Ito's conjecture pursued in~\cite{MinelliSourmelidisTechnau} rests crucially on the observation of Myerson~\cite[Lemma~5]{myerson1987semi-regular} that
\begin{equation}\label{eq:myersonobservation}
\ell(  x) = \varSigma_{\textrm{odd} }(x) - \epsilon(x)
\quad\text{and}\quad
\ell(1-x) = \varSigma_{\textrm{even}}(x) + \epsilon(x)
\end{equation}
for some\footnote{%
	Here $\varepsilon$ is some correction term which is related to our way of forcing uniqueness in the continued fraction expansion \cref{eq:continuedfractionexpansion}.
}
$\epsilon(x)\in\set{0,1}$, and on establishing the asymptotic formula
\[
\sum_{x\in \mathscr{F}_{1/2}(Q)} \ell(x)= c_1 Q^2(\log Q)^2 + c_2 Q^2\log Q + \LandauO\parentheses*{Q^2},
\]
where $2c_1 = C_1$ and $2c_2 = C_2+3/4$ with the constants $C_1$ and $C_2$ given by~\eqref{eq:Zhabitskaya:Constants}.
Consequently, obtaining asymptotic formulas similar to the one above is a crucial step in proving \cref{thm:Goal} and \cref{thm:ourresult2}.

To state our next theorem we introduce for a reduced rational number $\alpha=v/w$ in $(0,1]$ the quantity
\[
F(\alpha)
\coloneqq \frac{1}{\alpha w^2} \sum_{r<w} (1-\braces{\alpha r}) \zeta(2,r/w),
\]
where 
\[
\zeta(s,x) \coloneqq \sum_{n\geq0} \frac{1}{(n+x)^s}\quad
(\Re s>1, \, x\in\lopeninterval{0,1}),
\]
denotes the Hurwitz zeta-function.

\begin{thm}\label{thm:arbitrarycutsell}
	Let $Q$ be a positive integer and $\alpha=v/w$ be a reduced rational number in $(0,1]$ with $w<Q^{1/2}$.
	Let also $E(\alpha):=\ell(\alpha)+(\log w)^2$.
	Then
	\[
	\frac{1}{\sharp \mathscr{F}(Q)} \sum_{x\in\mathscr{F}_{\alpha}(Q)} \ell(x)
	= \frac{\alpha (\log Q)^2}{2\zeta(2)}
	+ \frac{\alpha  \log Q   }{ \zeta(2)} \parentheses*{
		2\gamma - 2 \frac{\zeta'(2)}{\zeta(2)} + F(\alpha) - \frac{3}{2}
	}
	+ \LandauO\parentheses*{E(\alpha)}.
	\]
	Moreover, if $\alpha<1$, then
	\[
	\frac{1}{\sharp \mathscr{F}(Q)} \sum_{x\in\mathscr{F}^\compl_{\alpha}(Q)} \ell(x)
	= \frac{1}{\sharp \mathscr{F}(Q)} \sum_{x\in\mathscr{F}_{\alpha}(Q)} \ell(x)
	- \parentheses*{ 1 - \frac{1}{w^2} } \log Q
	+ \LandauO\parentheses*{E(1-\alpha)}.
	\]
\end{thm}

\subsection{Structure of the paper}
In \cref{sec:auxiliaryresults} we collect several technical lemmas. 
In \cref{sec:proofsofthemainresults} we prove our results assuming \cref{prop:gapbetweensystems}, whose proof is given right after in \cref{sec:ProofOfMainProposition}.
Lastly, in \cref{sec:remarksandsomeheuristics} we conclude with some remarks on \cref{thm:ourresult2}.

\subsection{Notation}

We use the Landau notation $f(x) = \LandauO(g(x))$ and the Vinogradov notation $f(x) \ll g(x)$ to mean that there exists some constant $C>0$ such that $\abs{f(x)} \leq C g(x)$ holds for all admissible values of $x$ (where the meaning of `admissible' will be clear from the context).
Unless otherwise indicated, any dependence of $C$ on other parameters is specified using subscripts.
Similarly, we write `$f(x) = o(g(x))$ as $x\to\infty$' if $g(x)$ is positive for all sufficiently large values of $x$ and $f(x)/g(x)$ tends to zero as $x\to\infty$.

Given two coprime integers $a$ and $q\neq 0$ we write $\inv_q(a)$ for the least positive integer in the residue class $(a\bmod q)^{-1}$.

A sum with a star ($\Sum$) is understood to involve the additional summation condition \enquote{$\gcd(p,q) = 1$}. Sums like $\displaystyle\sum_{n\leq x}$ are understood to range over all $n=1,2,\ldots$ not exceeding $x$.
\section{Auxiliary Results}
\label{sec:auxiliaryresults}

\subsection{Lemmas on modular inversion}

We require information on the distribution of modular inverses.
The next lemma, itself a special case of a more general distribution result (see~\cite[Theorem 13]{shparlinskimodhyp}), which follows from Weil's bound on exponential sums, contains what we need.

\begin{lem}\label{lem:POMI}
	Let $p$ be a positive integer, $\alpha \in (0,1]$, $X_2 > X_1 \geq 0$.
	Then, for any $\epsilon>0$, both sums
	\[
	\Sum_{\substack{ X_1\leq q\leq X_2 \\ \inv_{p}(q) \leq \alpha p }} 1
	\quad\text{and}\quad
	\Sum_{\substack{ X_1\leq q\leq X_2 \\ \inv_{p}(q) > (1-\alpha) p }} 1
	\]
	are asymptotically equal to
	\[
	\alpha \frac{\varphi(p)}{p} (X_2 - X_1)
	+ \LandauO_{\epsilon}\parentheses*{ \frac{X_2 - X_1 + p}{p^{1/2-\epsilon}} }.
	\]
\end{lem}
\begin{proof}
	The proof is identical to that of \cite[Lemma~A.1]{MinelliSourmelidisTechnau} where the special case $\alpha=1/2$ is considered.
	We omit the details.
\end{proof}

The next result allows one to pass from $\inv_p(q)$ to $\inv_q(p)$.
A special case of this was already used in~\cite[Lemma~4.1]{MinelliSourmelidisTechnau}.

\begin{lem}\label{lem:inversiontrick}
	Let $p,q\geq2$ be coprime integers and $\alpha = {v}/{w}$ be a reduced rational number in $(0,1]$ with $w\leq p$ or $w\leq q$. 
	Then 
	\begin{align*}
		\inv_p(q) \leq \alpha p
		\quad\text{if and only if}\quad
		\inv_q(p)   >  (1-\alpha) q.
	\end{align*}
\end{lem}
\begin{proof}
	We start with the well-known identity
	\[
	q\inv_p(q)+p\inv_q(p)=1+pq
	\]
	(cf.\ \cite[proof of Lemma~4.1]{MinelliSourmelidisTechnau}).
	From this, by rearranging terms, we infer that
	\begin{align*}
		\inv_p(q)\leq  \alpha p
		\quad\text{if and only if}\quad
		\inv_q(p)\geq(1-\alpha)q+\frac{1}{p}.
	\end{align*}
	Similarly,
	\begin{align*}
		\inv_q(p)>(1-\alpha)q
		\quad\text{if and only if}\quad
		\inv_p(q)<  \alpha p +\frac{1}{q}.
	\end{align*}
	The result now easily follows by appealing to either of the two assumptions $w \leq p$ or $w \leq q$.
\end{proof}

\begin{lem}\label{lem:deltaplusminus}
	Let $\alpha = v/w$ be a reduced rational number in $(0,1]$.
	Suppose that
	\[
	\delta_\alpha(q) = \Sum_{(1-\alpha)q<p\leq q} 1
	\quad\text{and}\quad
	\kappa_\alpha(q) = \sum_{r<w}(1-\braces{\alpha r})\mathop{\sum_{d\mid q}}_{d\equiv r\bmod w}\mu\parentheses*{\frac{q}{d}}.
	\]
	Then, one has that $\delta_\alpha(q)=\alpha\varphi(q)+\kappa_\alpha(q)$ for every integer $q\geq1$.
\end{lem}
\begin{proof}
	We begin by rewriting the summation condition $\gcd(p,q)=1$ implicit in $\Sum$ using Möbius inversion.
	Hence,
	\[
	\delta_\alpha(q)
	= \sum_{(1-\alpha)q < p \leq q} \sum_{d\mid\gcd(p,q)} \mu(d)
	= \sum_{d\mid q} \mu(d) \sharp\set{ p \in \NN }[ (1-\alpha)q < p \leq q, \, d\mid p ].
	\]
	The sets on the right hand side have cardinality $\alpha q/d + \braces{(1-\alpha)q/d}$.
	Hence,
	\[
	\delta_\alpha(q)= \alpha \sum_{d\mid q} \mu(d) \frac{q}{d} + \sum_{d\mid q} \mu(d) \braces*{(1-\alpha)\frac{q}{d}}.
	\]
	The first sum on the right hand side equals $\varphi(q)$.
	Upon grouping the terms of the second sum by the residue $d\bmod w$, we obtain the function $\kappa_\alpha(q)$.
\end{proof}	

\subsection{Continued fractions}	
The key step in the approach of Lochs~\cite{lochs1961statistik} and Heilbronn~\cite{heilbronn1968average-length} to statistics of lengths of (ordinary) continued fraction expansions is the transference of the problem to a problem of counting solutions to certain Diophantine inequalities.
Here we require a version for minus continued fraction expansions.
Such a version was initially given by Zhabitskaya~\cite[Lemma~2]{zhabitskaya2009average-length}.
In~\cite[Lemma~4.3]{MinelliSourmelidisTechnau} the authors employed a restricted variant of this, which we generalize below in the case of arbitrary rational cuts.
\begin{lem}\label{lem:TheSumOfLengths}
	Let $\alpha$ be a rational number in $(0,1]$ and $Q$ be a positive integer.
	If $N_\alpha(Q)$ denotes the sum of the lengths of the minus continued fraction expansions of the numbers $a/b$ with $1\leq a<\alpha b$, $q\leq Q$, and $T_\alpha(Q)$ denotes the number of solutions $(a_1,q_1,a_2,q_2,m,n,a,b)\in\NN^8$ of the following system of equalities and inequalitites
	\begin{gather}\label{eq:TheSystem:1}
		\left\lbrace\begin{array}{@{}lll@{}}
			a_1q_2 - a_2 q_1 = 1, &
			1\leq a_1 \leq \alpha q_1, &
			1\leq a_2\leq q_2, \\
			n a_2 - m a_1 = a, &
			n q_2 - m q_1 = b, &
			1 \leq a < b \leq Q, \\
			1 \leq m < n, &
			1\leq q_1 < q_2,
		\end{array}\right.
	\end{gather}
	then
	\[
	N_\alpha(Q) = T_\alpha(Q) + \LandauO\parentheses*{\ell(\alpha)Q^2}.
	\]
\end{lem}
\begin{proof}
	Following \cite[Lemma 2]{zhabitskaya2009average-length} we know that $N_\alpha(Q)$ is, apart from an error $O\parentheses*{Q^2}$, equal to the number of solutions of
	\begin{gather}\label{eq:TheSystem:1.1}
		\left\lbrace\begin{array}{@{}lll@{}}
			a_1q_2 - a_2 q_1 = 1, &
			1\leq a_1 \leq q_1, &
			1 \leq a_2 \leq q_2, \\
			n a_2 - m a_1 = a, &
			n q_2 - m q_1 = b, &1\leq a<\alpha b\leq\alpha Q,\\
			1 \leq m < n, &
			1\leq q_1 < q_2.
		\end{array}\right.
	\end{gather} 
	In particular, the pairs $(p_1,q_1)$ and $(p_2,q_2)$ generated from each solution of the above system are successive convergents of the minus continued fraction expansion of $a/b$ other than $a/b$ itself \cite[Lemma 1]{zhabitskaya2009average-length}.
	Moreover, the sequence of these convergents is monotonically decreasing to $a/b$.
	Therefore, if $a_1\leq\alpha q_1$ then $a/b<a_1/q_1\leq\alpha$.
	
	It is not generally true that if $a/b$ is less than $\alpha$ then so is the case for its convergents.
	However, if $\ell(a/b)\geq\ell(\alpha)$, then the $\ell(\alpha)$-th convergent of $a/b$ will be less than or equal to $\alpha$, and so will be the succeeding convergents since they form a strictly decreasing sequence. 
	Indeed, if $\alpha=\rcfrac{1;t_1,\ldots,t_u}$ and $a/b=\rcfrac{1;s_1\ldots,s_v}<\alpha$ with $u\leq v$, then $\rcfrac{1;s_1,\dots,s_u}\leq\alpha$.
	For if $\rcfrac{1;s_1}\leq\alpha$ then we are done.
	Otherwise, the inequalities $a/b<\alpha<\rcfrac{1;s_1}$ imply that $s_1=t_1$. 
	If now $\rcfrac{1;t_1,s_2}\leq\alpha$ then we are done.
	Otherwise, the inequalities $a/b<\alpha<\rcfrac{1;t_1,s_2}$ imply that $s_2=t_2$.
	Repeating these steps at most $\ell(\alpha)$ times, we obtain the desired inequality.
	
	Hence, for each pair $(a,b)$, the number of solutions of the systems~\cref{eq:TheSystem:1} and~\cref{eq:TheSystem:1.1} differ by at most $\ell(\alpha)$.
	Since we count $O\parentheses{Q^2}$ pairs $(a,b)$, the lemma follows.
\end{proof}

The eight variables in~\cref{lem:TheSumOfLengths} can be reduced to four.
This is the contents of the next lemma.

\begin{lem}\label{lem:reducedN4system}
	Let $\alpha$ be a rational number in $(0,1]$ and $Q$ be a positive integer.
	If $R_\alpha(Q)$ denotes the number of solutions $(p, q, n, m) \in \NN^4$ of the following system of inequalities
	\begin{gather}\label{eq:TheSystem:CC}
		\left\lbrace\begin{array}{@{}lll@{}}
			\gcd(p,q) = 1, &
			p\geq2,&q \geq 1,\\
			2 \leq n q + kp \leq Q, &
			1 \leq k < n, \\
			\inv_p(q) \leq \alpha p,
		\end{array}\right.
	\end{gather}
	then
	\begin{align*}
		N_\alpha(Q) = R_\alpha(Q) + \LandauO\parentheses*{\ell(\alpha)Q^2}.
	\end{align*}
\end{lem}
\begin{proof}
	As in \cite[Lemma~4.4]{MinelliSourmelidisTechnau}, one shows that
	$T_\alpha(Q)=R_\alpha(Q)+\LandauO\parentheses*{Q^2}$.
\end{proof}

\subsection{Auxiliary asymptotic formulae}
We conclude this section with a list of formulae which will be employed in the succeeding proofs.
\begin{lem}\label{lem:MobInverFormula}
	Let $\Psi(Q) = a Q^2 (\log Q)^2 + b Q^2 \log Q + \LandauO(Q^2)$.
	Then
	\[
	\sum_{d\leq Q} \mu(d) \Psi\parentheses*{\dfrac{Q}{d}}
	= \frac{a}{\zeta(2)} Q^2 (\log Q)^2 + \frac{1}{\zeta(2)} \parentheses*{b-2a\dfrac{\zeta'(2)}{\zeta(2)}} Q^2 \log Q + \LandauO(Q^2).
	\]
\end{lem}
\begin{proof}
	For a proof see, e.g., \cite[Corollary~3]{zhabitskaya2009average-length}.
\end{proof}

\begin{lem} \label{lem:Euler'sPhiAsymptotics}
	The following asymptotic formulae hold for any $x\geq 2$:
	\begin{enumerate}
		\item
		\(\displaystyle
		\sum_{q< x}\varphi(q)=\frac{x^2}{2\zeta(2)}+O(x\log x)
		\),
		\item
		\(\displaystyle
		\sum_{q< x}\frac{\varphi(q)}{q}=\frac{x}{\zeta(2)}+O(\log x)
		\),
		\item\label{eq:Euler'sPhiAsymptotics3}
		\(\displaystyle
		\sum_{q< x}\frac{\varphi(q)}{q^2}=\frac{1}{\zeta(2)}\parentheses*{\log x+\gamma-\frac{\zeta'(2)}{\zeta(2)}}+O\parentheses*{\frac{\log x}{x}}
		\),
		\item\label{eq:Euler'sPhiAsymptotics4}
		\(\displaystyle
		\sum_{q< x}\frac{\varphi(q)}{q^2}\log q=\frac{\parentheses*{\log x}^2}{2\zeta(2)}+O(1)
		\).
	\end{enumerate}
	If in addition $\alpha\coloneqq v/w$ is a reduced rational in the unit interval and $\kappa_\alpha$ is defined as in \cref{lem:deltaplusminus}, then for $x>w$
	\[
	\sum_{q<x}\frac{\kappa_\alpha(q)}{q^2}=\sum_{r<w}\frac{(1-\braces*{\alpha r})\zeta\parentheses*{2,r/w}}{w^2\zeta(2)}+\LandauO\parentheses*{\frac{\log x}{x}}.
	\]
\end{lem}
\begin{proof}
	The first two formulae are well known and the proof of the third one can be found in \cite[Corollary~4.5]{boca2007products-of-matrices}.
	Formula~\cref{eq:Euler'sPhiAsymptotics4} can be deduced easily from~\cref{eq:Euler'sPhiAsymptotics3} and partial summation.
	The last formula follows directly from the definitions of $\kappa_\alpha(q)$ and $\zeta(s,x)$, and computations of well-known harmonic sums via the Euler-Maclaurin summation formula.	
\end{proof}

\begin{lem}\label{lem:SomeMoreAsymptoticFormulae}
	The following asymptotic formulae hold for any $U\geq2$:
	\begin{enumerate}
		\item\(\displaystyle
		\mathop{\sum_k\sum_q}_{k+q<U}\frac{\varphi(q)}{kq^2}
		=\frac{\parentheses*{\log U}^2}{\zeta(2)}+\frac{\log U}{\zeta(2)}\parentheses*{2\gamma-\frac{\zeta'(2)}{\zeta(2)}}+O(1)
		\),
		\item\(\displaystyle
		\mathop{\sum_k\sum_q}_{k+q<U}\frac{\varphi(q)}{q^2}
		=\frac{U\log U}{\zeta(2)}+O(U)
		=\mathop{\sum_k\sum_q}_{k+q<U}\frac{\varphi(q)}{qk}
		\),
		\item\(\displaystyle
		\mathop{\sum_k\sum_q}_{k+q<U}\frac{\varphi(q)k}{q^2}
		=\frac{U^2\log U}{2\zeta(2)}+O\parentheses*{U^2}
		=\mathop{\sum_k\sum_q}_{k+q<U}\frac{\varphi(q)}{k}
		\).
	\end{enumerate}
\end{lem}
\begin{proof}
	They follow directly from the formulae of the previous lemma.
\end{proof}
\section{Proofs of the main results}\label{sec:proofsofthemainresults}

In this section we state our main technical result (\cref{prop:gapbetweensystems} below) and show how to deduce \cref{thm:arbitrarycutsell} from it.
We then proceed on proving \cref{thm:Goal} and \cref{thm:ourresult2}.
The proof of \cref{prop:gapbetweensystems} is given in~\cref{sec:ProofOfMainProposition}.

The main step is counting the number of solutions to the systems~\cref{eq:TheSystem:CC} 
asymptotically.
For technical reasons, it turns out to be convenient to group these solutions according to the size of their components and count the number of solutions in each group individually.
To fix the relevant notation, let $U=Q^{1/2}$, and consider the following five cases:
\begin{itemize}
	\item $p \leq q \leq U$;             \hfill (`Case~1')
	\item $p \leq q$, $U < q$;           \hfill (`Case~2')
	\item $q < p \leq U$;                \hfill (`Case~3')
	\item $q < p$, $U < p$, $n \leq U$;  \hfill (`Case~4')
	\item $q < p$, $U < p$, $U < n$.     \hfill (`Case~5')
\end{itemize}

\begin{prop}\label{prop:gapbetweensystems}
	Let $\alpha={v}/{w}$ be a reduced rational number in $(0,1]$ and $Q\geq 1$ be real.
	For $i=1,2,3,4,5$, let $R_{i,\alpha}(U)$ 
	denote the number of solutions $(p, q, n, m)$ to the systems~\cref{eq:TheSystem:CC} 
	subject to the additional constraint that \enquote{Case~$i$} be satisfied. 
	Then we have the following asymptotic formulae:
	\item
	\begin{enumerate}
		\item $R_{1,\alpha}(U)=\dfrac{\alpha \log2}{2\zeta(2)}U^4\log U+O\parentheses*{U^4}$,
		\item $R_{2,\alpha}(U)=\dfrac{\alpha \log 2}{2\zeta(2)}U^4\log U+O\parentheses*{U^4}$,
		\item $\dfrac{R_{3,\alpha}(U)}{U^4}=\dfrac{\alpha (\log U)^2}{4\zeta(2)}+\dfrac{\alpha \log U}{2\zeta(2)}\parentheses*{\gamma-\dfrac{\zeta'(2)}{\zeta(2)}+F(\alpha)-\log 2}+\LandauO\parentheses*{(\log w)^2}$,
		\item $R_{4,\alpha}(U)=\dfrac{\alpha U^4 (\log U)^2}{4\zeta(2)}+ \dfrac{\alpha U^4\log U}{2\zeta(2)}(\gamma-\log 2)+\LandauO\parentheses*{U^4}$,
		\item $R_{5,\alpha}(U)=\dfrac{\alpha U^4(\log U)^2}{2\zeta(2)}+\dfrac{\alpha U^4\log U}{2\zeta(2)}\parentheses*{2\gamma-\dfrac{\zeta'(2)}{\zeta(2)}+F(\alpha)-3}+\LandauO\parentheses*{U^4}$.
	\end{enumerate}
\end{prop}

Assuming the conclusion of \cref{prop:gapbetweensystems} for the moment, we are now in position to prove our main results.

\begin{proof}[Proof of \cref{thm:arbitrarycutsell}]
	Let $Q\geq 1$ be an integer and put $U = Q^{1/2}$.
	From \cref{lem:reducedN4system} and \cref{prop:gapbetweensystems} we know that
	\begin{align*}
		N_\alpha(Q)
		&=\sum_{i\leq 5}R_{i,\alpha}(U)+\LandauO(\ell(\alpha)Q^2)\\
		&=\frac{\alpha Q^2(\log Q)^2}{4\zeta(2)}+\frac{\alpha Q^2\log Q}{2\zeta(2)}\parentheses*{2\gamma-\frac{\zeta'(2)}{\zeta(2)}+F(\alpha)-\frac{3}{2}}+\LandauO\parentheses*{E(\alpha)Q^2}.
	\end{align*}
	Moreover, by M\"obius inversion,
	\[
	\sum_{x\in \mathscr{F}_{\alpha}(Q)} \ell(x)=\sum_{d\leq Q}\mu(d)N_\alpha\parentheses*{\frac{Q}{d}}.
	\]
	The first formula of the theorem follows now from the above, \cref{lem:MobInverFormula} and the well-known formula $\sharp \mathscr{F}(Q)={Q^2}/\parentheses*{2\zeta(2)}+\LandauO(Q)$.
	
	The second formula of the theorem follows from the first one by observing that
	\[
	\frac{1}{\sharp \mathscr{F}(Q)} \sum_{x\in\mathscr{F}^\compl_{\alpha}(Q)} \ell(x)=\frac{1}{\sharp \mathscr{F}(Q)} \parentheses*{\sum_{x\in\mathscr{F}_{1}(Q)} \ell(x)-\sum_{x\in\mathscr{F}_{1-\alpha}(Q)} \ell(x)+\LandauO(\ell(1-\alpha))}
	\]
	and that
	\[
	F(1)-(1-\alpha)F(1-\alpha)=-(1-\alpha)F(1-\alpha)=\alpha F(\alpha)-\zeta(2)\parentheses*{1-\frac{1}{w^2}}.
	\qedhere
	\]
\end{proof}

\begin{proof}[Proof of \cref{thm:Goal}]
	If $x$ has continued fraction expansion $x = [0; a_1, \dots, a_n]$, then, by~\cref{eq:hickersonformula},
	\(
	s(x) = \varSigma_{\pm}(x)/12 + \LandauO(1)
	\),
	where the implied constant is absolute.
	In view of~\cref{eq:myersonobservation} we obtain that
	\[
	12\sum_{x\in\mathscr{F}_\alpha(Q)}s(x)+\LandauO\parentheses*{Q^2}= \sum_{x\in\mathscr{F}_\alpha(Q)} (\ell(x)-\ell(1-x))=\sum_{x\in\mathscr{F}_\alpha(Q)}\ell(x)-\sum_{x\in\mathscr{F}^c_\alpha(Q)}\ell(x).
	\]
	The theorem follows now from~\cref{thm:arbitrarycutsell}.
\end{proof}
\begin{proof}[Proof of \cref{thm:ourresult2}]
	Let 
	\[
	\mathscr{S}\coloneqq\set*{\alpha=[0;a_1,a_2,\dots]\in[0,1]\setminus\mathbb{Q}:\sum_{m\leq M}a_m\ll_\alpha\phi^{M/2}},
	\]
	where $\phi$ is the golden ratio.
	One can easily show by standard techniques from the metric theory of continued fractions (see for example \cite{khinchin1964continued-frac}), that the set $\mathscr{S}$ is of full Lebesgue measure.
	
	Let $\alpha\in\mathscr{S}$ and $f(Q)$ be a monotonically increasing positive function such that $f(Q)=o(\log Q)$ as $Q\to\infty$.
	If $p_n/q_n\coloneqq[0;a_1,\dots,a_n]$ denotes the $n$-th convergent of $\alpha$, then we know that
	\[
	\liminf_{n\to\infty}\frac{\log q_n}{n}\geq\log\phi.
	\]
	We employ this property and the one describing the set $\mathscr{S}$ to prove \cref{thm:ourresult2}.
	
	Indeed, let $Q$ be any sufficiently large positive integer. 
	There is a positive integer $N=N(Q)$, which we assume w.l.o.g. that is even, such that
	\begin{align}\label{goldenratio}
		\phi^{N-1}<\parentheses*{f(Q)\log Q}^{1/2}\leq\phi^{N}\leq q_N.
	\end{align}
	Observe that in view of~\cref{eq:myersonobservation} 
	\[
	\ell\parentheses*{\frac{p_N}{q_N}}+\ell\parentheses*{1-\frac{p_N}{q_N}}<\ell\parentheses*{\frac{p_{N+1}}{q_{N+1}}}+\ell\parentheses*{1-\frac{p_{N+1}}{q_{N+1}}}=\sum_{m\leq N+1}a_m\ll_\alpha \phi^{N/2}
	\]
	and
	\[
	\log q_N<\log q_{N+1}\leq\log\prod_{m\leq N+1}(a_n+1)\ll_\alpha \phi^{N/2}.
	\]
	Therefore, the error terms appearing in \cref{thm:arbitrarycutsell} for the case of $\beta\coloneqq p_N/q_N$ and $\beta'\coloneqq p_{N+1}/q_{N+1}$ are all bounded from above by $\LandauO_\alpha\parentheses*{\parentheses*{f(Q)\log Q}^{1/2}}$.
	
	Since $\beta$ and $\beta'$ are successive convergents of $\alpha$, we know that $\beta'-\beta=1/(q_Nq_{N+1})$. 
	In addition to relation \cref{goldenratio} and \cref{thm:arbitrarycutsell} we obtain that
	\begin{align*}
		\Delta(\beta',\beta)
		&\coloneqq\frac{1}{\sharp \mathscr{F}(Q)} \parentheses*{\sum_{x\in\mathscr{F}_{\beta'}(Q)}\ell(x)-\sum_{x\in\mathscr{F}^\compl_{\beta}(Q)}\ell(x)}\\
		&=\parentheses*{1+\frac{\beta'F(\beta')-\beta F(\beta)}{\zeta(2)}}\log Q+\LandauO_\alpha\parentheses*{\frac{\log Q}{f(Q)}+\parentheses*{f(Q)\log Q}^{1/2}}.
	\end{align*}
	However,
	\begin{align*}
		\beta' F(\beta')-\beta F(\beta)
		&=\sum_{r<q_{N+1}}\frac{\parentheses*{1-\braces{\beta' r}}}{q_{N+1}^2}\parentheses*{\frac{q_{N+1}^2}{r^2}+O(1)}+\\
		&\quad-\sum_{r<q_{N}}\frac{\parentheses*{1-\braces{\beta r}}}{q_{N}^2}\parentheses*{\frac{q_{N}^2}{r^2}+O(1)}\\
		&=\sum_{r<q_N}\frac{\braces*{\beta r}-\braces{\beta'r}}{r^2}+\LandauO\parentheses*{\frac{1}{q_N}},
	\end{align*}
	where the last sum is also $\LandauO\parentheses*{1/q_{N}}$ as can be seen from $\braces*{\beta r}-\braces{\beta'r}=-r/(q_Nq_{N+1})$.
	
	In conclusion, we deduce that
	\[
	\Delta(\beta',\beta)=\log Q+\LandauO_\alpha\parentheses*{\frac{\log Q}{f(Q)}+\parentheses*{f(Q)\log Q}^{1/2}}.
	\]
	A similar computation shows that $\Delta(\beta,\beta')$ has the same asymptotics as $\Delta(\beta',\beta)$.
	The theorem follows now from relation
	\[
	\frac{12}{\sharp\mathscr{F}(Q)}\sum_{x\in\mathscr{F}_\alpha(Q)}s(x)=\Delta(\alpha,\alpha')
	\]
	and the inequalities $\beta<\alpha<\beta'$ and $\Delta(\beta,\beta')\leq\Delta(\alpha,\alpha)\leq\Delta(\beta',\beta)$.
\end{proof}
\section{Proof of \texorpdfstring{\cref{prop:gapbetweensystems}}{Proposition\autoref{prop:gapbetweensystems}}}
\label{sec:ProofOfMainProposition}
	
This section is devoted to the proof of \cref{prop:gapbetweensystems}, where we follow the same procedure as in~\cite{MinelliSourmelidisTechnau}.
Therefore, we heavily refer to~\cite{MinelliSourmelidisTechnau} to shorten certain parts of the proof.

\subsection*{Case 1}
We count the number of solutions $R_{1,\alpha}(U)$ of
\[
	\left\lbrace
	\begin{array}{@{}ll@{}}
	\gcd(p,q) = 1,& 2\leq p\leq q\leq U, \\
	2 \leq n q + kp \leq U^2,& 1\leq k<n,\\
	\inv_p(q)\leq \alpha p.
	\end{array}
	\right.
\]
Following \cite[§~5, Case~1]{MinelliSourmelidisTechnau} with $\alpha p$ in place of $p/2$, we find that
\[
	R_{1, \alpha}(U)=\frac{U^4}{2}\mathop{\sum_{2\leq p\leq U}\Sum_{p\leq q\leq U}}_{\inv_{p}(q)\leq \alpha p}\frac{1}{q(p+q)}+O\parentheses*{U^3}=\frac{\alpha \log2}{2\zeta(2)}U^4\log U+O\parentheses*{U^4}.
\]
\subsection*{Case 2}
We count the number of solutions $R_{2, \alpha}(U)$ of
\begin{gather}\label{eq:TheSystem:CC2-gen}
	\left\lbrace
	\begin{array}{@{}lll@{}}
	\gcd(p,q) = 1, &
	2\leq p\leq q,&U<q, \\
	2 \leq n q + kp \leq U^2, &
	1\leq k<n,\\
	\inv_p(q)\leq \alpha p.\\
	\end{array}.
	\right.
\end{gather}
Observe that in this case the inequalities $n\leq U^2/q<U$ hold as well. 
	
Consider the pairs $(n,k)\in\mathbb{N}^2$ for which $1\leq k<n$ and $n+k>U$. 
In that case the number of solutions of the above system is smaller than the number of solutions of the same system without the restrictions on coprimality and modular inversion. 
This number has been estimated in \cite[(54)--(56)]{zhabitskaya2009average-length} to be $\LandauO(U^4)$.
	
Let $\mathscr{C}\coloneqq\set{(p,q)\in\ZZ^2}[\gcd(p,q)=1]$.
It suffices to consider the number of solutions of the system~\eqref{eq:TheSystem:CC2-gen} with the additional restriction $n+k\leq U$.
If we fix such a pair $(n,k)$, then the domain of solutions of \eqref{eq:TheSystem:CC2-gen} can be expressed as the lattice
\[
	\mathscr{L}_{1,\alpha}(n,k)
	= \set*{ (p,q)\in\mathscr{C} }[
	2\leq p\leq\frac{U^2}{n+k},\,
	U<q\leq\frac{U^2-kp}{n}, \,
	\inv_p(q)\leq \alpha p
	]
\]
without the lattice
\[
	\mathscr{L}_{2,\alpha}(n,k)
	= \set*{ (p,q)\in\mathscr{C} }[
	U<p\leq\frac{U^2}{n+k}, \,
	U<q\leq p, \,
	\inv_p(q)\leq \alpha p
	].
\]
The number of integer points of $\mathscr{L}_1(n,k)$ and $\mathscr{L}_2(n,k)$ is equal to 
\[
	S_{1,\alpha}(n,k) \coloneqq\sum_{p\leq U^2/(n+k)} A_{(p, \alpha)}\parentheses*{U,\frac{U^2-kp}{n}},
\]
and
\[
	S_{2,\alpha}(n,k)\coloneqq\sum_{U<p\leq U^2/(n+k)} A_{(p,\alpha)}(U,p),
\]
respectively, where $A_{(p,\alpha)}(x,y)$ is defined as in \cref{lem:POMI}.
By following \cite[§~5, Case~2]{MinelliSourmelidisTechnau} with $A_{(p,\alpha)}$ in place of $A_{(p,1/2)}$, we obtain that
\[
	R_{2,\alpha}(U)=\alpha \frac{\log 2}{2\zeta(2)}U^4\log U+O\parentheses*{U^4}.
\]	
\subsection*{Case 3}
We count the number of solutions $R_{3,\alpha}(U)$ of 
\[
	\left\lbrace\begin{array}{@{}ll@{}}
	\gcd(p,q) = 1, &
	1\leq q<p\leq U, \\
	2 \leq n q + kp \leq Q, &
	1 \leq k < n,\\
	\inv_p(q) \leq \alpha p,
	\end{array}\right..
\]
Similar to Case~1 (see also \cite[(58)--(60)]{zhabitskaya2009average-length}), the number of solutions of the above systems are, up to an error term of order $\LandauO\parentheses*{U^3\log U}$, equal to
\begin{align}\label{eq:TheSolutionsofCC3-gen}
	\frac{U^4}{2}\mathop{\sum_{p\leq U}\Sum_{ q<p}}_{\inv_{p}(q)\leq \alpha p}\frac{1}{q(p+q)}.
\end{align}
We rewrite the first double sum above as
\begin{align*}
	&\mathop{\sum_{w<p\leq U}\Sum_{ p^{1/2}\leq q<p}}_{\inv_{p}(q)\leq \alpha p}\frac{1}{pq}+\mathop{\sum_{w<p\leq U}\Sum_{ q<p^{1/2}}}_{\inv_{p}(q)\leq \alpha p}\frac{1}{pq}+\mathop{\sum_{p\leq w}\Sum_{q<p}}_{\inv_{p}(q)\leq \alpha p}\frac{1}{pq}-\mathop{\sum_{p\leq U}\Sum_{q<p}}_{\inv_{p}(q)\leq \alpha p}\frac{1}{p(p+q)}\\
	&\eqqcolon S_1(U)+S_2(U)+S(\alpha)-S_4(U).
\end{align*}
	
By partial summation and \cref{lem:POMI} we obtain for fixed $\epsilon\in(0,1/2)$ that
\begin{align}\label{eq:case3-gen-S4}
	\begin{aligned}
	S_4(U)
	=\sum_{p\leq U}\frac{1}{p}\brackets*{\frac{A_{(p,\alpha)}\parentheses*{1,p}}{2p}+\int_{1}^p\frac{A_{(p,\alpha)}\parentheses*{1,u}\mathrm{d}{u}}{(u+p)^2}}=\frac{\alpha\log 2\log U}{\zeta(2)}+O(1).
	\end{aligned}
\end{align}
To estimate $S_1(U)$ we will first apply~\cref{lem:inversiontrick}, obtaining
\[
	S_1(U)=\sum_{w^{1/2}<q<U}\frac{1}{q}\mathop{\Sum_{v_q<p\leq V_q}}_{\inv_{q}(p)> (1-\alpha )q}\frac{1}{p}
	,
\]
where $v_q\coloneqq\max\braces{w,q}$ and $V_q\coloneqq\min\braces*{U,q^2}$.
Then partial summation and \cref{lem:POMI} yield that
\[
	\mathop{\Sum_{v_q<p\leq V_q}}_{\inv_{q}(p)> (1-\alpha )q}\frac{1}{p}
	\begin{multlined}[t]
	= \frac{B_{(q,\alpha)}(v_q,V_q)}{V_q} + \int_{v_q}^{V_q}\frac{B_{(q,\alpha)}(v_q,u)\mathrm{d}u}{u^2} \\
	= \frac{\alpha\varphi(q)}{q} \log\frac{V_q}{v_q} + \LandauO_\epsilon\parentheses*{q^{-1/2+2\epsilon}} 
	\end{multlined}
\]
Therefore, for fixed $\epsilon\in(0,1/4)$, we deduce that
\begin{align*}
	S_1(U)
	&=\sum_{w^{1/2}<q\leq w}\frac{\alpha\varphi(q)}{q^2}\log\frac{q^2}{w}+\sum_{w<q\leq U^{1/2}}\frac{\alpha\varphi(q)}{q^2}\log q+{}\\
	&\quad+\sum_{U^{1/2}<q< U}\frac{\alpha\varphi(q)}{q^2}\log\frac{U}{q}+\LandauO(1)\\
	&=\sum_{q\leq U^{1/2}}\frac{\alpha\varphi(q)}{q^2}\log q+\sum_{U^{1/2}<q< U}\frac{\alpha\varphi(q)}{q^2}\log\frac{U}{q}+\LandauO\parentheses*{(\log w)^2}.
\end{align*}
	
For $S(\alpha)$ we have the trivial bound $\LandauO\parentheses{\parentheses{\log w}^2}$.
	
It remains to estimate $S_2(U)$ to which we will apply first \cref{lem:inversiontrick}.
Observe here that to do so we will have to remove from $S_2(U)$ the sum  corresponding to $q=1$. 
This sum is always included in $S_2(U)$ since $p>w$.
Therefore,
\[
	S_2(U)=\sum_{w<p\leq U}\frac{1}{p}+\sum_{2\leq q< U^{1/2}}\frac{1}{q}\mathop{\Sum_{u_q<p\leq U}}_{\inv_{q}(p)> (1-\alpha )q}\frac{1}{p},
\]
where $u_q\coloneqq\max\braces*{w,q^2}$.
Since
\[
	\sharp\set{p\leq x}[p\equiv\inv_{p}(b)\bmod q]=\frac{x}{q}+O(1)
\]
for any coprime integers $1\leq b\leq q$, we obtain from partial summation that
\[
	\sum_{\substack{u_q<p\leq U\\ p\equiv \inv_q(b)}} \frac{1}{p}=\frac{1}{q} \log \frac{U}{u_q}+\LandauO\parentheses*{q^{-2}}.
\]
Hence,
\begin{align*}
    S_{2}(U)
	&=\log\frac{U}{w}+\LandauO(1)+\sum_{2\leq q<U^{1/2}}\frac{1}{q}\,\Sum_{(1-\alpha)q<b\leq q}\parentheses*{\frac{1}{q} \log \frac{U}{u_q}+\LandauO\parentheses*{q^{-2}}}\\
	&=\log\frac{U}{w}+\sum_{2\leq q\leq w^{1/2}}\frac{\delta_\alpha(q)}{q^2}\log\frac{U}{w}+\sum_{w^{1/2}<q<U^{1/2}}\frac{\delta_\alpha(q)}{q^2}\log\frac{U}{q^2}+\LandauO(1).
\end{align*}
where $\delta_\alpha(q)\leq q$ was defined in \cref{lem:deltaplusminus}.
In view of the aforementioned lemma, we conclude that
\begin{align*}
	S_2(U)
	&=\log{U}\sum_{q\leq U^{1/2}}\frac{\delta_\alpha(q)}{q^2}-2\sum_{q<U^{1/2}}\frac{\delta_\alpha(q)}{q^2}\log q+\LandauO\parentheses*{(\log w)^2}\\
	&=\log{U}\sum_{q\leq U^{1/2}}\frac{\alpha\varphi(q)}{q^2}-2\sum_{q<U^{1/2}}\frac{\alpha\varphi(q)}{q^2}\log q+{}\\
	&\quad+\log{U}\sum_{q\leq U^{1/2}}\frac{\kappa_\alpha(q)}{q^2}-2\sum_{q<U^{1/2}}\frac{\kappa_\alpha(q)}{q^2}\log q+\LandauO\parentheses*{(\log w)^2}.
\end{align*}
Observe that $\kappa_\alpha(q)\leq d(q)$, where $d(n)$ denotes the number of divisors of $n$.
Therefore, the last sum in the last line above is $\LandauO(1)$ and, thus,
\[
	S_1(U)+S_2(U)=\sum_{q\leq U}\frac{\alpha\varphi(q)}{q^2}\log \frac{U}{q}+\log{U}\sum_{q\leq U^{1/2}}\frac{\kappa_\alpha(q)}{q^2}+\LandauO\parentheses*{(\log w)^2}.
\]
It follows now from \cref{lem:Euler'sPhiAsymptotics} and relations \eqref{eq:TheSolutionsofCC3-gen} and \eqref{eq:case3-gen-S4} that
\[
	R_{3,\alpha}(U)=\frac{\alpha U^4(\log U)^2}{4\zeta(2)}+\frac{\alpha U^4\log U}{2\zeta(2)}\parentheses*{\gamma-\frac{\zeta'(2)}{\zeta(2)}+F(\alpha)-\log 2}+\LandauO\parentheses*{U^4(\log w)^2}.
\]
\subsection*{Case 4}
We count the number of solutions $R_{4,\alpha}(U)$ of
\begin{gather}\label{eq:TheSystem:CC4-gen}
	\left\lbrace\begin{array}{@{}lll@{}}
	\gcd(p,q) = 1, &
	1\leq q<p,& U<p, \\
	2 \leq n q + kp \leq U^2, &
	1\leq k<n\leq U,\\
	\inv_p(q)\leq \alpha p.
	\end{array}\right.
\end{gather}
Similar to Case~2, we fix $k$ and $n$, and count the number of solutions of the above system, when $n+k\leq U$, and when $n+k>U$.
	
If $n+k>U$, then the number of solutions of \eqref{eq:TheSystem:CC4-gen} is smaller than the number of solutions of the same system without the restrictions on coprimality and modular inversion.
This number has been computed in \cite[(64)--(65)]{zhabitskaya2009average-length} to be $\LandauO\parentheses*{U^4}$.
	
If now $n+k\leq U$, then the domain of solutions of \eqref{eq:TheSystem:CC4-gen} can be expressed as the union of lattices
\[
	\mathscr{L}_{1,\alpha}(n,k)
	= \set*{(p,q)\in\mathscr{C}}[
	U< p\leq\frac{U^2}{n+k},\,
	1\leq q<p, \,
	\inv_p(q)\leq \alpha p
	]
\]
and
\begin{align*}
	\mathscr{L}_{2,\alpha}(n,k) &
	= \set*{(p,q)\in\mathscr{C}}[
	\frac{U^2}{n+k}<p\leq\frac{U^2}{k}, \,
	2\leq q\leq\frac{U^2-k p}{n}, \,
	\inv_p(q)\leq \alpha p
	] \\ &
	= \set*{(p,q)\in\mathscr{C}}[
	2\leq q\leq\frac{U^2}{n+k}-\theta,\,
	\frac{U^2}{n+k}<p\leq\frac{U^2-nq}{k},\,
	\inv_q(p)> (1-\alpha)q
	]
\end{align*}
without the lattice of integer points
\[
	\mathscr{L}_{3,\alpha}(n,k) = \set*{ (p,1) }[ \frac{U^2}{n+k}<p\leq\frac{U^2}{k} ],
\]
where we have also applied \cref{lem:inversiontrick} in $\mathscr{L}_{2,\alpha}(n,k)$  and have introduced a parameter $\theta\in[0,1]$ which may vary.
The number of points of $\mathscr{L}_1(n,k)$, $\mathscr{L}_3(n,k)$ and $\mathscr{L}_2(n,k)$ is equal to 
\[
	S_{1,\alpha}(n,k) \coloneqq\sum_{U<p\leq U^2/(n+k)} A_{(p, \alpha)}\parentheses*{1,p},\quad S_{3,\alpha}(n,k) \coloneqq\sum_{U^2/(n+k)<p\leq U^2/k}1,
\]
and
\[
	S_{2,\alpha}(n,k)\coloneqq\sum_{2\leq q\leq U^2/(n+k)-\theta} B_{(q,\alpha)}\parentheses*{\frac{U^2}{n+k}\frac{U^2-nq}{k}}.
\]
respectively.
Working as in~\cite[§~5, Case~4]{MinelliSourmelidisTechnau}, and using \cref{lem:POMI}, we obtain that
\begin{align*}
	S_{1,\alpha}(n,k)
	&=\sum_{U<p\leq U^2/(n+k)} \frac{\alpha\varphi(p)}{2}+\LandauO_\epsilon\parentheses*{\frac{U^2}{n+k}\sum_{U<p\leq U^2/(n+k)}\frac{1}{p^{1/2-\epsilon}}}\\
	&=\frac{\alpha U^4}{2\zeta(2)(n+k)^2}+\LandauO_\epsilon\parentheses*{U^2+\frac{U^2}{n+k}\log\frac{U^2}{n+k}+\frac{U^{3+2\epsilon}}{(n+k)^{3/2+\epsilon}}},
\end{align*}
\begin{align*}
	S_{2,\alpha}(n,k)=\frac{\alpha n U^4}{2\zeta(2)k(n+k)^2}+\LandauO_\epsilon\parentheses*{\frac{nU^2}{k(n+k)}\log\frac{U^2}{n+k}+\frac{nU^{3+2\epsilon}}{k(n+k)^{3/2+\epsilon}}}
\end{align*}
and
\[
	S_{3,\alpha}(n,k)\ll\frac{nU^2}{k(n+k)}.
\]
Therefore, by fixing $\epsilon\in(0,1/4)$, we conclude that
\begin{align*}
	R_{4,\alpha}(U)=\frac{\alpha }{4\zeta(2)}U^4 (\log U)^2+ \frac{\alpha}{2\zeta(2)}(\gamma-\log 2)U^4\log U+\LandauO\parentheses*{U^4}.
\end{align*}
\subsection*{Case 5}
We count the number of solutions $R_{5, \alpha}(U)$ of
\[
	\left\lbrace
	\begin{array}{@{}lll@{}}
	\gcd(p,q) = 1, &
	1\leq q<p,&U<p, \\
	2 \leq n q + kp \leq U^2, &
	1\leq k<n,&U<n,\\
	\inv_p(q)\leq \alpha p,
	\end{array}
	\right.
\]
which, apart from the number of solutions corresponding to $q=1$, is equivalent to
\begin{equation}\label{eq:TheSystem:CC52-gen-inv}
	\left\lbrace
	\begin{array}{@{}lll@{}}
	\gcd(p,q) = 1, &
	2\leq q<p,&U<p, \\
	2 \leq n q + kp \leq U^2, &
	1\leq k<n,&U<n,\\
	\inv_q(p)>(1-\alpha) q,
	\end{array}
	\right.
\end{equation}
as can be seen from \cref{lem:inversiontrick}.
We may restrict ourselves to the case $k+q<U$ for, otherwise, the system \eqref{eq:TheSystem:CC52-gen-inv} has no solution.
If $S(k,q)$ denotes the number of solutions of \eqref{eq:TheSystem:CC52-gen-inv} for fixed $k$ and $q$, then
\begin{equation}\label{eq:R51}
	R_{5,\alpha}(U)=\sum_{k<U-1}\sum_{U<n\leq U^2-k\ceil{U}}\sum_{U< p\leq\parentheses*{U^2-n}/{k}}1+\mathop{\sum_{k}\sum_{q\geq2}}_{k+q<U}S(k,q),
\end{equation}
where $\ceil{x}$ is the ceiling function.
	
A straightforward computation yields 
\begin{equation}\label{eq:R5major}
	\sum_{k<U-1}\sum_{U<n\leq U^2-k\ceil{U}}\sum_{U< p\leq\parentheses*{U^2-n}/{k}}1=\frac{U^4}{2k}+\LandauO\parentheses*{U^3}.
\end{equation}
On the other hand
\begin{align*}
	S(k,q)
	&=\sum_{U<n\leq\left(U^2-k\lceil U\rceil\right)/q}\mathop{\sum_{(1-\alpha)q<b\leq q}}_{\gcd(b,q)=1}\mathop{\sum_{ U<p\leq \left(U^2-nq/k\right)}}_{p\equiv \inv_{q}(b)} 1\\
	&=\sum_{U<n\leq\parentheses*{U^2-k\lceil U\rceil}/q}\mathop{\sum_{(1-\alpha) q<b\leq q}}_{\gcd(b,q)=1}\parentheses*{\frac{1}{q}\parentheses*{\frac{U^2-nq}{k}-U}+O(1)}\\
	&=\sum_{U<n\leq\parentheses*{U^2-k\lceil U\rceil}/q}\parentheses*{\frac{\delta_\alpha (q)}{q}\parentheses*{\frac{U^2-nq}{k}-U}+O\parentheses*{q}}.
\end{align*}
Recall that $\kappa_\alpha(q)\leq d(q)$.
We then have in view of \cref{lem:deltaplusminus} that, for $q>1$,
\begin{align*}
	S(k,q)
	&=\frac{\alpha\varphi(q)}{q}\sum_{U<n\leq\parentheses*{U^2-k\lceil U\rceil}/q}\parentheses*{\frac{U^2-nq}{k}-U}+\frac{\kappa_\alpha(q)}{kq}\sum_{n\leq\parentheses*{U^2-k\lceil U\rceil}/q}\parentheses*{{U^2-nq}}+\\
	&\quad+\LandauO\parentheses*{U^2+\frac{U^3d(q)}{q^2}+\frac{U^3d(q)}{kq}}\\
	&\eqqcolon\Sigma(k,q)+\frac{\kappa_\alpha(q)U^4}{2kq^2}+\LandauO\parentheses*{U^2+\frac{U^3d(q)}{q^2}+\frac{U^3d(q)}{kq}}.
\end{align*}
Similarly as in \cite[§~5, Case~5]{MinelliSourmelidisTechnau} with $\alpha$ instead of $1/2$, we obtain for $q>1$
\[
	\Sigma(k,q)=\frac{\alpha\varphi(q)U^4}{2kq^2}-\frac{\alpha\varphi(q)U^3}{q^2}-\frac{\alpha\varphi(q)U^3}{kq}+\frac{\alpha\varphi(q)kU^2}{2q^2}+\frac{\alpha\varphi(q)U^2}{2k}+\LandauO\parentheses*{U^2}.
\]
	
From~\cref{eq:R51} and~\cref{eq:R5major}, we infer that
\begin{align*}
	R_{5,\alpha}(U) &
	= \sum_{k< U-1}\parentheses*{\frac{U^4}{2k}-\frac{\alpha U^4}{2k}-\frac{U^4(1-\alpha)}{2k}}
	+ \frac{\alpha U^4}{2}\mathop{\sum_{k}\sum_{q}}_{k+q<U}\frac{\varphi(q)}{kq^2} +{} \\ & \qquad
	+ \frac{U^4}{2}\mathop{\sum_{k}\sum_{q}}_{k+q<U}\frac{\kappa_\alpha(q)}{kq^2} 
	- \alpha U^3\mathop{\sum_{k}\sum_{q}}_{k+q<U}\parentheses*{\frac{\varphi(q)}{q^2}+\frac{\varphi(q)}{kq}} +{} \\ & \qquad
	+ \frac{\alpha U^2}{2}\mathop{\sum_{k}\sum_{q}}_{k+q<U}\parentheses*{\frac{\varphi(q)k}{q^2}+\frac{\varphi(q)}{q}}+\LandauO\parentheses*{U^4}.
\end{align*}
Each of the above sums is computed in \cref{lem:Euler'sPhiAsymptotics} and \cref{lem:SomeMoreAsymptoticFormulae}.
Thus, we conclude that
\[
	R_{5,\alpha}(U)=\frac{\alpha U^4(\log U)^2}{2\zeta(2)}+\frac{\alpha U^4\log U}{2\zeta(2)}\parentheses*{2\gamma-\frac{\zeta'(2)}{\zeta(2)}-3+F(\alpha)}+\LandauO\parentheses*{U^4}.
\]
\section{Remarks and some heuristics}\label{sec:remarksandsomeheuristics}

In this final section we discuss briefly the result we obtained. In particular, we want to formalize the intuition that the main contribution to the sample at irrational cuts is carried by the points which are \enquote{arbitrarily} close to $0$. To this end, we borrow an elementary argument from \cite{conrey1996mean-values}.

\begin{lem}\label{lem:pointsclosetozero}
	Let $g\colon\RR\to\RR$ be a monotonically increasing function with $\lim\limits_{x\to \infty} g(x)=\infty$ and define the set 
	\begin{align*}
		\mathscr{S}(g,Q)=\left\{\frac{a}{b} \in \mathbb{Q}:  b\leq Q\,\text{ and  }\, \frac{a}{b}\leq \frac{1}{g(Q)}\right\}.
	\end{align*}
	Then we have that
	\begin{align*}
		\sum_{\frac{a}{b} \in \mathscr{S}(g,Q)} s(b,a) \ll \frac{Q^2}{g(Q)^2}\log^2(Q).
	\end{align*}
\end{lem}

\begin{proof}
	We start by writing
	\[
		\sum_{\frac{a}{b} \in \mathscr{S}(g,Q)} s(b,a)=\sum_{a<Q/g(Q)}\sum_{ag(Q)<b<Q} s(b,a)
	\]
	Since $s(b,a)=s\parentheses*{\bar{b}, a}$, where $\bar{b}$ denotes the reduction of $b$ modulo $a$, relation~\cref{eq:SumOfDedekindSums} implies that
	\begin{align*}
		\sum_{ag(Q)<b<Q} s(b,a)\ll\max_{S} \Big\vert \sum_{k\in S} s(k, a)\Big\vert,
	\end{align*} 
	where the maximum ranges over all subsets $S$ of the set of the residue classes modulo $a$.
	Moreover, it is known (see for example the introduction of \cite{girstmair2005arithmetic-mean}) that
	\begin{align*}
		\Sum_{k<a} \vert s(k,a)\vert \ll a\log^2 a.
	\end{align*}
	Using the above facts we conclude that
	\[
		\sum_{a<Q/g(Q)}\sum_{ag(Q)<b<Q} s(b,a)\ll \sum_{a<Q/g(Q)}a\log^2 a\ll \frac{Q^2}{g(Q)^2}\log^2(Q).
		\qedhere
	\]
\end{proof}

\begin{lem}
	Let $g(x)=\log x$ and $\mathscr{S}(g, Q)$ be as in \cref{lem:pointsclosetozero}. 
	Then
	\begin{align*}
		\frac{1}{\mathscr{F}(Q)}\sum_{\frac{a}{b} \in \mathscr{S}(g,Q)} s(a,b)=\frac{1}{12}\log Q+O\left(\log\log Q\right).
	\end{align*}
\end{lem}
\begin{proof}
	By the reciprocity law for Dedekind sums
	\[
	s(a,b)+s(b,a)=\frac{1}{12}\parentheses*{\frac{b}{a}+\frac{a}{b}+\frac{1}{ab}}-\frac{1}{4}
	\]
	and \cref{lem:pointsclosetozero} we obtain that
	\begin{align*}
		\sum_{\frac{a}{b} \in \mathscr{S}(g,Q)} s(a,b)=\frac{1}{12}\sum_{\frac{a}{b} \in \mathscr{S}(g,Q)} \left(\frac{b}{a}+\frac{a}{b}+\frac{1}{ab}\right)+\LandauO\parentheses*{Q^2}
		=\frac{1}{12} \sum_{\frac{a}{b} \in \mathscr{S}(g,Q)}  \frac{b}{a}+\LandauO\parentheses*{Q^2}.
	\end{align*}
	The lemma follows by M\"obius inversion and computing the resulting double sum.
\end{proof}

\end{document}